\documentclass[a4paper,12pt]{amsart}
\usepackage[T2A]{fontenc}
\usepackage[cp1251]{inputenc}
\usepackage[english]{babel}

\usepackage{graphics}
\usepackage{epsfig}
\usepackage{graphicx}
\usepackage{amssymb}
\usepackage{epsf}

\advance\hoffset-20mm \advance\textwidth40mm

\advance\voffset-15mm \advance\textheight30mm

\theoremstyle{plain}
\newtheorem{theo}{Theorem}
\newtheorem{lemma}{Lemma}[section]
\newtheorem{proposition}[lemma]{Proposition}
\newtheorem{corollary}[lemma]{Corollary}
\theoremstyle{definition}
\newtheorem{definition}[lemma]{Definition}
\newtheorem{remark}[lemma]{Remark}
\newtheorem{example}[lemma]{Example}

\makeatletter
\newcounter{primer}[section]
0

\newcommand{\prr}{\par\smallskip\refstepcounter{primer}\xdef\@currentlabel{\thesection.\arabic{primer}}%
\textbf{Counterexample \arabic{section}.\arabic{primer}.} } 

\newcounter{sluchaj}[section]
0

\newcommand{\slu}{\par\smallskip\refstepcounter{sluchaj}\xdef\@currentlabel{\thesection.\arabic{sluchaj}}%
\textbf{Case \arabic{section}.\arabic{sluchaj}.} } 
\makeatother

\newcounter{line}
0

\def\kk{{\Bbbk}}

\def\ZZ{{\mathbb Z}}

\def\QQ{{\mathbb Q}}
\def\Qg0{{\mathbb Q_{\geqslant 0}}}

\def\Lie{{\rm Lie}}

\def\Spin{{\rm Spin}}

\def\conv{{\rm conv}}

\def\vol{{\rm vol}}

\def\qmatrix#1{\left(\begin{array}{*{20}r}#1\end{array}\right)}

\makeatletter
\def\kratno{\mathbin{\lower.3ex\hbox{$\m@th\vdots$}}}
\makeatother

\newcounter{itemnumber}

\begin{document}
\title[SIMPLE MODULES WITH NORMAL CLOSURES OF MAXIMAL TORUS ORBITS]{SIMPLE MODULES OF CLASSICAL LINEAR GROUPS \\ WITH NORMAL CLOSURES OF MAXIMAL TORUS ORBITS}

\author%
{K. Kuyumzhiyan}

\thanks{This research benefited from the support of the <<EADS Foundation Chair in Mathematics>>, RFBR grant 09-01-00648a, Russian-French Poncelet Laboratory (UMI 2615 of CNRS), and Dmitry Zimin fund <<Dynasty>>.}

\address{Department of Algebra, Faculty of Mechanics
and Mathematics, Moscow State University, Leninskie Gory 1,
GSP-1, Moscow, 119991, Russia}
\email{karina@mccme.ru}


\begin{abstract}
Let $T$ be a maximal torus in a classical linear group~$G$.
In this paper we find all simple rational $G$-modules $V$ such that
for each vector $\mathsf{v}\in V$ the closure of its $T$-orbit is a normal
affine variety. For every other $G$-module we present a $T$-orbit with
the non-normal closure. We use a combinatorial criterion of normality
formulated in terms of the set of weights of a simple $G$-module. This work is a continuation of~\cite{KK},
where the same problem is solved in the case $G=SL(n)$.
\end{abstract}


\maketitle


\section*{Introduction} 

Let~$G$ be an affine algebraic group over an algebraically closed field~$\kk$ of characteristic zero, and let~$G$ act on an affine algebraic 
variety. Recall that an irreducible affine algebraic variety~$X$ is called {\it normal} if its algebra of regular functions $\kk[X]$ is 
integrally closed in its field of fractions. 
The study of normality of orbit closures has a long history. The first results were obtained by Kostant~\cite{Kost}. He showed that 
for a reductive group~$G$ the full nilpotent cone in the adjoint module is normal. 
Kraft and Procesi~\cite{KP1} proved that in the adjoint module~$\mathfrak{sl(n)}$ 
the closures of all $SL(n)$-orbits are normal. 
The analogous result for~$SL(n)$ over a field of positive characteristic was established by Donkin~\cite{Do}. Later, Kraft and 
Procesi~\cite{KP2} and Sommers~\cite{So2} studied the same question for the adjoint modules of other classical groups. In particular, 
the orbits with non-normal closures are constructed in~\cite{KP2}, in the language of Young diagrams. 
The cases~$F_4$, $G_2$, $E_6$ are considered by Broer, Kraft, and Sommers in~\cite{Br}, \cite{Kraft}, and~\cite{So1}. There 
is no complete answer for~$E_7$ and~$E_8$ yet.

Now let us consider actions of an algebraic torus~$T$, i.e., of an affine algebraic group isomorphic to 
$\kk^\times\times\ldots\times \kk^\times$, where $\kk^\times=\kk\setminus \{0\}$. 
An irreducible algebraic variety~$X$ is called \textit{toric} if it is normal and if it admits a regular $T$-action with an open orbit. 
Toric varieties play an important role in algebraic geometry, topology, and combinatorics, since they can be completely described in terms 
of convex geometry, see e.g.~\cite{Ful}. If an algebraic torus~$T$ acts on a variety~$Y$, then the orbit closure $X=\overline{Ty}$ 
of a point $y\in Y$ is a natural candidate to be a toric variety. To verify it, one should check that~$X$ is normal. 

The normality property for the closure of a $T$-orbit in a module has a well-known combinatorial interpretation. Let $v_1,\ldots,v_r$ be vectors of a vector space~$\QQ^n$.
For any set $A$ of rational numbers we denote by $A(v_1,\ldots,v_r)$ the set of all linear combinations of vectors $v_1,\ldots,v_r$ with coefficients in $A$. 
The set $\{v_1,\ldots,v_r\}$ is called {\it saturated} if  
$$
\ZZ_{\geqslant 0} (v_1, v_2, \dots, v_r) = \ZZ (v_1, v_2, \dots,
v_r) \cap \Qg0 (v_1, v_2, \dots, v_r).
$$
In these terms, the closure of a $T$-orbit of a vector in the rational $T$-module is normal if and only if 
the set of weights in the weight decomposition of this vector is saturated.

Let $G$ be a connected semisimple algebraic group over an algebraically closed field $\kk$ of characteristic zero. Let $T\subseteq G$ be a fixed 
maximal torus. Consider a finite-dimensional rational $G$-module $V$. We are seeking for  modules~$V$ with the following property: 
for each vector $\mathsf{v}\in V$ the closure of its $T$-orbit $\overline{T\mathsf{v}}$ is a normal (affine) algebraic variety.

Earlier this property was checked by J. Morand~\cite{JM} in the case when $G$ is a simple group and $V$ is its adjoint module;
this problem for $G=SL(n)$ was also considered in~\cite[Ex. 3.7]{Stu} and \cite{Stumf}. In her previous paper~\cite{KK} the author
checks this property for all simple $SL(n)$-modules. For exceptional root systems, this problem is studied in~\cite{BK}.

The aim of this paper is to investigate this property for all simple modules of the special orthogonal group~$SO(n)$, the spinor group~$\Spin(n)$, and 
the symplectic group~$Sp(2n)$. All the proofs use the language of root systems. Recall that a simple $G$-module is uniquely defined by its highest weight~$\lambda$.
Any dominant weight $a_1\pi_1+\ldots+a_r\pi_r$ can play the role of~$\lambda$, where~ $\pi_1,\ldots,\pi_r$ stand for the fundamental weights, and
$a_1,\ldots,a_r$ are nonnegative integers. We enumerate the fundamental weights as in~\cite[Section 4]{VO}.

\begin{theo}\label{mth}
For the classical algebraic groups of types $B_n$, $C_n$, and $D_n$, the following modules, and their dual ones, are the 
only modules where the closures of all maximal torus orbits are normal. 
\begin{center}
\begin{tabular}{|c|c|c|}
\hline
Root system & Highest weight & Checked in \\
\hline 
$B_n, n\geqslant 2$&$\pi_1$ & \rm Case $\ref{line1}$\\
\hline
$B_2$&$\pi_2$ & \rm Case $\ref{sluchaj3}$\\
\hline
$B_2$ & $2\pi_2$ & \rm Case $\ref{line9}$\\
\hline
$B_3$&$\pi_3$ & \rm Case $\ref{sluchaj3}$\\
\hline
$B_4$&$\pi_4$ & \rm Case $\ref{sluchaj3}$\\
\hline
$C_n, n\geqslant 3$&$\pi_1$ & \rm Case $\ref{line8}$\\
\hline
$C_3$&$\pi_2$ & \rm Case $\ref{sluchaj5}$\\
\hline
$C_4$ &$\pi_2$ & \rm Case $\ref{sluchaj5}$\\
\hline
$D_n, n\geqslant 4$&$\pi_1$ &  \rm Case $\ref{line2}$\\
\hline
$D_4$ & $\pi_2$ & \rm Case $\ref{line12}$\\
\hline
$D_4$&$\pi_3$ & \rm Case $\ref{line4}$\\
\hline
$D_4$&$\pi_4$ & \rm Case $\ref{line4}$\\
\hline
$D_5$&$\pi_4$ & \rm Case $\ref{line6}$\\
\hline
$D_6$&$\pi_5$ & \rm Case $\ref{line7}$\\
\hline
$D_6$&$\pi_6$ & \rm Case $\ref{line7}$\\
\hline
\end{tabular}
\end{center}
In all the other cases, the module contains a maximal torus orbit with the non-normal closure.
\end{theo}

For a simple $G$-module $V(\lambda)$ with the highest weight $\lambda$ we denote by $M(\lambda)$ the set of all its weights with respect to the
maximal torus~$T$. Now, the closures of all $T$-orbits in the module $V(\lambda)$ are normal if and only if all subsets in~$M(\lambda)$ are saturated, i.e. if~$M(\lambda)$ is \textit{hereditarily normal}. 

The plan of the paper is the following. In the first section, we recall some necessary facts about the 
weight decomposition and modules with the given highest weights, and we also introduce some combinatorial notions 
concerning saturated sets. In sections 2--4, we prove Theorem~\ref{mth} for the root systems~$B_n$, $C_n$, and~$D_n$, 
respectively. We check hereditary normality of sets~$M(\lambda)$ for the weights~$\lambda$ listed in Theorem~\ref{mth} 
(positive cases), and in every other (negative) case we indicate a non-saturated subset. 
In the most difficult positive cases, we use the properties of unimodular sets of vectors and their generalizations. In negative cases, it is enough 
to treat minimal with respect to inclusion sets of weights, which are not listed in Theorem~$\ref{mth}$. 
Let us note that the most difficult positive cases are of independent interest as combinatorial facts. 

The author is grateful to her scientific supervisor I.V. Arzhantsev for
the formulation of the problem and fruitful discussions. Thanks are also due to I.I. Bogdanov for
useful comments, to R.A. Devyatov for the computer-based check of the most difficult cases, and to A.Yu. Novoseltsev for important remarks.

\section{Preliminaries}
We always denote by $e_1,\ldots,e_n$ the standard basis in~$\QQ^n$. The fractional part of a real value~$q$ is denoted by $\{q\}$, the integer part is denoted by $\lfloor q \rfloor$. The sign~$\kratno$ stands for divisibility, i.e. $a\kratno b \iff \exists c\in\ZZ$, $a=bc$.

\subsection{Weight decomposition}

Let $T$ be an algebraic torus and let $\Lambda=\Lambda(T)$ be the lattice of its $T$-characters.
For every rational $T$-module $V$ we have its weight decomposition
$$
V=\bigoplus_{\mu\in \Lambda} V_\mu, \quad\text{where}\quad V_\mu=\{\mathsf{v}\in V \,|\, t\mathsf{v}=\mu(t)\mathsf{v}\}.
$$
Denote by $M(V)=\{ \mu\in \Lambda \,|\, V_\mu\ne 0 \}$ the set of weights of the module~$V$.
For every nonzero vector $\mathsf{v}$ we have its weight decomposition $\mathsf{v}=\mathsf{v}_{\mu_1}+\dots+\mathsf{v}_{\mu_s}$, 
$\mathsf{v}_{\mu_i}\in V_{\mu_i}$, $\mathsf{v}_{\mu_i}\ne 0$.
Below, we consider weights $\mu\in \Lambda$ as points of the rational vector space $\Lambda_{\QQ} := \Lambda\otimes_\ZZ \QQ$.

The following statement is a well-known combinatorial criterion of normality of a $T$-orbit closure in a $T$-module, 
see~\cite[I, \S~1, Lemma~1]{KKMSD}.

\begin{proposition}
Let $V$ be a finite-dimensional rational $T$-module and $\mathsf{v}=\mathsf{v}_{\mu_1}+\dots+\mathsf{v}_{\mu_s}$ be the weight decomposition of a vector $\mathsf{v}\in V$.
The closure $\overline{T\mathsf{v}}$ of the $T$-orbit of $\mathsf{v}$ is normal if and only if the set of characters $\{\mu_1, \dots, \mu_s\}$ is saturated.
\end{proposition}

A finite subset $M$ of a rational vector space $\QQ^n$ is called {\it hereditarily normal} if all its subsets are saturated.

\begin{corollary}
Let $V$ be a finite-dimensional rational $T$-module. The closures of all $T$-orbits in the module $V$ are normal if and only if
the set $M(V)$ is hereditarily normal.
\end{corollary}

Notice that for the dual module $V^*$ one has $M(V^*)= - M(V)$.
This means that the property of hereditary normality for the set $M(V)$ is equivalent to the same property for the set~$M(V^*)$.

\subsection{Representations with the given highest weight}
Let $G$ be a connected simply connected semisimple algebraic group, let~$B$ be a Borel subgroup in~$G$, and let $T\subset B$ be the maximal torus.
Denote by~$\Phi$ the root system of the Lie algebra~$\Lie(G)$ associated with the maximal torus~$T$. Let~$\Phi^+$ and
$\Delta=\{\alpha_1,\ldots,\alpha_r\}\subseteq \Phi^+$, respectively, be the subsets of positive roots and of simple roots, 
corresponding to the Borel subgroup~$B$. Denote by~$\pi_i$ the fundamental 
weight corresponding to the simple root~$\alpha_i$. It is well-known that the weights $\pi_1,\ldots,\pi_r$ form a basis of the character lattice $\Lambda(T)$
of the torus~$T$. The semigroup generated by the fundamental weights coincides with the semigroup of dominant weights~$\Lambda_+$. The subgroup of~$\Lambda$ generated 
by the root system~$\Phi$ is called the {\it root lattice}, we denote it by~$\Xi$. Then $\Xi$ is the sublattice of~$\Lambda$ of finite index, and
$\alpha_1,\ldots,\alpha_r$ form a basis of $\Xi$.

Let $V(\lambda)$ be a simple $G$-module with the highest weight $\lambda\in\Lambda_+$. Recall the following description of the set of $T$-weights of the 
module~$V(\lambda)$. Let $W$ be the Weyl group of the root system~$\Phi$. Then $W$ can be realized as a finite group of linear transformations
of the vector space~$\Lambda_\QQ$ generated by reflections $s_{\alpha}$, where $\alpha$ is a root, see~\cite{Hum}. The {\it weight polytope} $P(\lambda)$ of the module $V(\lambda)$ is the convex hull 
$\conv\{w\lambda \,|\, w\in W\}$ of the $W$-orbit of the point $\lambda$ in $\Lambda_\QQ$. Then
$$
M(\lambda)=(\lambda+\Xi)\cap P(\lambda),
$$
see~\cite[Theorem 14.18]{FH}. There is a partial order on the vector space~$\Lambda_\QQ$: $\lambda \succeq \mu $  if and only if $\lambda-\mu$
is a linear combination of simple roots with nonnegative integer coefficients. We use the following classical lemma.

\begin{lemma}\label{l.ovl}
Let $\lambda,\; \lambda'\in \Lambda_+$. Suppose that $\lambda \succeq \lambda'$. Then $M(\lambda)\supseteq M(\lambda')$.
\end{lemma}

\begin{proof}
Use the criterion from~\cite[Exercice 1 to Section VIII, \S 7]{Bour}: the weight $\lambda'\in \lambda+\Xi$ belongs to $M(\lambda)$ if and only if for all 
$w\in W$ the weight $\lambda-w\lambda'$ is a linear combination of simple roots with nonnegative integer coefficients. First notice that under our assumptions $\lambda'$ belongs to $M(\lambda)$. 
Indeed, for $w=e$ the weight $\lambda-\lambda'$ is a linear combination of simple roots with nonnegative integer coefficients due to the assumption, 
and for $w\ne e$ it is known that $w\lambda'=\lambda'-\mu$, where~$\mu$
is a sum of positive roots. Hence $\lambda-w\lambda'=\lambda-\lambda'+\mu$ is a linear combination of simple roots with nonnegative integer coefficients. It means that 
$\lambda'\in M(\lambda)$, and all points of the form $w\lambda'$, where $w\in W$, belong to~$M(\lambda)$. Using convexity, we obtain that
$M(\lambda')\subseteq M(\lambda)$.
\end{proof}

\begin{corollary}\label{ovlozhenii}
Let $\lambda'\in \Lambda_+$, and assume that $M(\lambda')$ is not hereditarily normal. Then for all $\lambda \in \Lambda_+$ such 
that~$\lambda \succeq \lambda'$ the set $M(\lambda)$ is not hereditarily normal. 
\end{corollary}

\smallskip

\subsection{Non-saturated sets}\label{sec:unimod} 
The proof of the following lemma can be found in~\cite{JM}.

\begin{lemma}\label{l.oln}
Let $M\subset \QQ^n$ be a finite set of vectors.
\begin{enumerate}
\item[(i)]
If $M$ is linearly independent, then $M$ is saturated.
\item[(ii)] 
If $M$ is not saturated and contains both vectors $v$ and $-v$, then either $M\backslash\{ v\}$ or
$M\backslash \{-v\}$ is not saturated.
\item[(iii)]
Let $v\in\Qg0(M)$.
Then there exists a linearly independent subset $M'\subseteq
M$ such that $v\in\Qg0(M')$. 
\end{enumerate}
\end{lemma}

We refer to a nonsaturated subset as an {\it NSS}. By an {\it extended nonsaturated subset} we mean a nonsaturated subset $\{v_1, \dots, v_r\}$
together with a vector $v_0$ such that 
$$
v_0 \in (\ZZ (v_1, v_2, \dots,
v_r) \cap \Qg0 (v_1, v_2, \dots, v_r))\setminus  \ZZ_{\geqslant 0} (v_1, v_2, \dots, v_r),
$$
and such that there exists a $\Qg0$-combination
$$
v_0=q_1v_{i_1}+\ldots+q_sv_{i_s}, \quad v_{i_j}\in \{v_1, v_2, \dots, v_r\}
$$
with linearly independent vectors $v_{i_1},\ldots,v_{i_s}$ and coefficients $q_i\in [0,1)$. 

These subsets will be named {\it ENSS}s and will be denoted by~$\{v_0; v_1, \dots, v_r\}$.

\begin{lemma}\label{lemma01}
Suppose that a set $M=\{v_1, \dots, v_r\}$ is not saturated. Then there exists a vector~$v_0$ such 
that~$\{v_0; v_1, \dots, v_r\}$ is an ENSS.
\end{lemma}

\begin{proof}
Consider any vector $v_0\in (\ZZ(M)\cap \Qg0(M)) \setminus \ZZ_{\geqslant 0} (M)$, and the corresponding $\Qg0$-combination 
${v_0=q_1v_1+\ldots+q_rv_r}$. By Lemma~\ref{l.oln}\,(iii) there exists a linearly independent subset $\{v_{i_1},\ldots,v_{i_s}\}\subseteq \{v_1,\ldots,v_r\}$ and the collection of $\Qg0$-coefficients $q'_j$ such that $v_0=q'_1v_{i_1}+\ldots+q'_sv_{i_s}$. If some
$q'_j\geqslant 1$, consider another vector $v_0'=v_0-\lfloor q'_1\rfloor v_{i_1}-\dots-\lfloor
q'_s\rfloor v_{i_s}$ instead of $v_0$. It is easy to see that it also belongs to $\ZZ(v_1,\dots,
v_r)$ and to $\Qg0(v_1,\dots, v_r)$, and does not belong to
$\ZZ_{\geqslant 0}(v_1,\dots, v_r)$. However all the coefficients of the new $\QQ_{\geqslant
0}$-combination belong to the semiopen interval $[0,1)$. This means that $\{v_0'; v_1,\dots,
v_r\}$ is an~ENSS.
\end{proof}

Let $v_0, v_1, \dots, v_r$ be some vectors in the vector space~$\QQ^n$, and let $f$ be a linear function on~$\QQ^n$. We call $f$ a {\it discriminating linear function} for the collection 
$\{v_0; v_1, \dots, v_r\}$ if the value $f(v_0)$ cannot be represented as a linear combination of values $f(v_1),\ldots$, $f(v_r)$ with nonnegative integer coefficients. 
If it is known that $v_0$ belongs to $\ZZ (v_1, v_2, \dots, v_r) \cap \Qg0 (v_1, v_2, \dots, v_r)$ and that it can be represented as a
$\Qg0$-combination of linearly independent vectors $v_1, \dots, v_r$ with coefficients from the semiopen interval~$[0,1)$, then the existence of a discriminating function guarantees 
that $\{v_0; v_1, \dots, v_r\}$ is an~ENSS.

\subsection{Unimodular and almost unimodular sets}
Assume that a set of vectors $M\subset \QQ^n$ has rank $d$, $d\leqslant n$, and $L=\langle v \,|\, v\in M \rangle$ is the linear span
of vectors from~$M$. The set $M$ is called \textit{unimodular} if for every linearly independent vectors
$v_1,\ldots,v_d\in M$ the value of the $d$-dimensional volume $\vol_d (v_1,v_2,\ldots,v_d)$ has constant absolute value. 
If one fixes a basis in~$L$, then the condition above is equivalent to the fact that absolute values of all nonzero determinants 
$|\det (v_1,v_2,\ldots,v_d)|$, $v_1,v_2,\ldots,v_d\in M$, computed in this basis are equal.

If the set $M$ is unimodular and $M_1\subseteq M$ is a subset, then the intersection of~$M$ with the subspace $L_1\subset L$, 
$L_1=\langle v \,|\, v\in M_1 \rangle$, is also unimodular. It can be easily seen after choosing a basis in~$L$ compatible with~$L_1$.

The next theorem is used in many proofs.

\begin{theo}[{\cite[Thm.~3.5]{Stu}}]\label{lemmaunimodsverhnas}
Any unimodular set of vectors $M$ is hereditarily normal.
\end{theo}

We say that a subset $M\subset \QQ^n$ of rank $d$ is {\it almost unimodular} if we can choose a subset $\{v_1,v_2,\ldots,v_d\}\subseteq M$ such that in some basis of the space $\langle M \rangle$
$$
\det(v_1,v_2,\ldots,v_d)=m,
$$
and for every other vector $v'\in M$ and for each~$i$ the value
$$
\det(v_1,v_2,\ldots,\widehat{v_i},\ldots,v_d,v')
$$  
equals $km$ for some $k\in \ZZ$. 
The value $m=\det(v_1,v_2,\ldots,v_d)$ is called the {\it volume} of the almost unimodular subset. 
By a {\it primitive subset} $\{v_1,\ldots,v_d\}$ of an almost unimodular set of volume~$m$ we mean a subset such that its determinant 
equals~$\pm m$. The property that the set is almost unimodular and its primitive subsets do not depend on the choice of basis 
in~$\langle M \rangle$.

\begin{lemma}\label{kramer}
Consider an almost unimodular set $M$ such that all determinants in $M$ are contained in the set $m\cdot\{1,a_1,\ldots,a_k\}$ and for some vectors $w_1,\ldots,w_d\in M$ the value $\det(w_1,\ldots,w_d)$ equals $am$. If we decompose a vector $w\in M$ in the basis $w_1,\ldots,w_d$, then the coefficients belong to the set $\{\pm 1/a, \pm a_1/a,\ldots,\pm a_k/a\}$.
\end{lemma}
\begin{proof}
Let us expand a vector $w\in M$ in the basis $(w_1,w_2,\ldots,w_d)$. By Cramer's formulae, it has 
the following coordinates:
$$
\text{if }w=b_1w_1+\ldots+b_dw_d,\;
\text{then }b_i=\frac{\det(w_1,w_2,\ldots,\widehat{w_i},w,\ldots,w_d)}{\det(w_1,w_2,\ldots,w_d)}, 
$$
i.e. all the $b_i$s have the form $\{\pm 1/a, \pm a_1/a,\ldots,\pm a_k/a\}$.
\end{proof}

\begin{corollary}\label{razlozhenie}
For every primitive subset $v_1,\ldots,v_d\subseteq M$, the set $M$ belongs to $\ZZ(v_1,\ldots,v_d)$.
\end{corollary}

\begin{corollary}\label{razloz}
In an almost unimodular set~$M$ of volume~$m$ and rank~$d$, the values of all the determinants have the form~$km$, $k\in \ZZ$. 
\end{corollary}
\begin{proof}
Let $w_1,\ldots,w_d\in M$. We have: $\det(w_1,\ldots,w_d)=\det A\cdot\det (v_1,\ldots,v_d)$, where~$A$ is an integer matrix 
expressing the vectors $w_1,\ldots,w_d$ in the basis $(v_1,v_2,\ldots,v_d)$. Since $\det A\in \ZZ$, the value $\det(w_1,\ldots,w_d)$ has the desired form.
\end{proof}

This gives an equivalent definition of an almost unimodular set: it is a set in which all the determinants are divisible by some~$m$, and 
there exists a determinant which equals exactly~$m$.

\begin{example}
Consider the set $M$ containing $16$ points $\{(\pm1,\pm1,\pm
1,\pm1)\}$. It is easy to see that determinants of all $4$-tuples equal $0$,
$8$, or $16$. This means that $M$ is almost unimodular.
\end{example}

\begin{lemma}\label{nemin} 
Suppose that an almost unimodular set $M$ of rank $d$ and of volume $m$ is not hereditarily normal, and $\{v_0; v_1,\ldots,v_r\}$ is an ENSS. 
Assume that the corresponding
$\Qg0$-combination for $v_0$ involves only the linearly independent vectors $v_1,\ldots,v_{d'}$.
\begin{enumerate}
\item[(i)]
If $d' = \dim \langle v_1,\ldots,v_{d'} \rangle = d$, then $|\vol_d ( v_1,\ldots,v_d)| \ne m$. 
 \item[(ii)]
If $d' < d$, then for any vectors $w_{d'+1}, \ldots, w_d \in M$ linearly independent with $ v_1,\ldots,v_{d'}$ one has $|\vol_d ( v_1,\ldots,v_{d'}, w_{d'+1}, \ldots, w_d )| \ne m$.
 \end{enumerate}
\end{lemma}

\begin{proof}
(i)
If $|\vol_d ( v_1,\ldots,v_{d'})| = m$, then by Corollary~$\ref{razlozhenie}$ the vector $v_0$ decomposes with integer coefficients in the basis $v_1,\ldots,v_d$.
Since $v_1,\ldots,v_d$ are linearly independent, this $\ZZ$-combination coincides with the initial $\Qg0$-combination, a contradiction. 

(ii)
We may suppose that vectors $w_{d'+1}, \ldots, w_d$ occur in the initial $\Qg0$-combination for $v_0$ with zero coefficients, and then use the reasoning of
the previous part.
\end{proof}
\subsection{The ratio of determinants is two}
In this section we consider an almost unimodular set~$M$ of volume~$m$ such that all its nonzero determinants equal~$\pm m$ or~$\pm 2m$.
\begin{lemma}\label{12unimod}
Consider an almost unimodular set~$M$ of rank~$d$ such that all its nonzero determinants equal~$\pm m$ or~$\pm 2m$, and 
suppose that~$M$ is not hereditarily normal. Let $\{v_0; v_1,\ldots,v_r\}$ be the corresponding ENSS, and let 
$v_0=q_1v_1+\ldots+q_lv_l$ be the corresponding $\Qg0$-combination. Then all $q_i\in \{0,1/2\}$.
\end{lemma}

\begin{proof}
Denote by $d'$ the rank of the set $\{v_1,\ldots, v_r\}$. %
Complete $\{v_1,\ldots,v_l\}$ to a basis of the space~$\langle M \rangle$. Now the statement follows from Lemmas~\ref{nemin}\,(i), 
\ref{kramer}, and the definition of an ENNS. 
\end{proof}

In the next three lemmas, we fix a basis $(\bar v_1$, $\bar v_2,\ldots$, $\bar v_d)$ of volume~$2m$. 
By Lemma~\ref{kramer}, the other vectors of~$M$ will be decomposed in this basis with coefficients~$0$, $\pm 1/2$, and~$\pm 1$. 
For every vector~$v \in M$, denote by $S(v)$ the set of indices corresponding to coordinates~$\pm 1/2$. 

\begin{lemma}\label{proSv}
If $S(v_1)\neq\emptyset$ and $S(v_2)\neq\emptyset$, then $S(v_1)=S(v_2)$.
\end{lemma}
\begin{proof}
Suppose that $S(v_1)\neq S(v_2)$ and $\# S(v_1)\geqslant\# S(v_2)$. Choose indices $i\in S(v_1)\setminus S(v_2)$ and $j\in S(v_2)$, $j\neq i$. 
An easy check shows that  
$$
\det (\bar v_1,\bar v_2,\ldots,\widehat{\bar v_i},\ldots,\widehat{\bar v_j},\ldots,\bar v_d, v_1, v_2)\in \left\{\pm \frac m2, \pm \frac{3m}2, \pm \frac{5m}2\right\},
$$ 
a contradiction.
\end{proof} 

\begin{lemma}\label{peremesh}
Suppose that a finite group~$W$ acts by permutations on a set~$M$ and linearly in~$\langle M \rangle$ in such a way that for every basis  
$(\bar v_1$, $\bar v_2,\ldots$, $\bar v_d)$ of volume~$2m$ and two indices $i$, $j$ there exists a $w\in W$, permuting the lines 
$\langle \bar v_1 \rangle$, $\langle \bar v_2 \rangle, \ldots$, $\langle \bar v_d \rangle$, and interchanging $\langle \bar v_i \rangle$ 
and $\langle \bar v_j \rangle$. Then for every $v\in M$, all its nonzero coordinates in the basis $(\bar v_1$, $\bar v_2,\ldots$, $\bar v_d)$ either are in the set $\{\pm 1\}$, or in the set $\{\pm 1/2\}$.
\end{lemma}
\begin{proof}
On the contrary, let~$v$ be such that it has~$\pm 1$ on the $i$th position, and $\pm 1/2$ on the $j$th position. 
Interchanging the lines $\langle \bar v_i \rangle$ and $\langle \bar v_j \rangle$, we obtain a vector $w(v)$ of $M$, whereas
 $S(v)\neq S(w(v))$, and it contradicts Lemma~\ref{proSv}. 
\end{proof}

\begin{lemma}\label{resultm2m}
Consider an almost unimodular set~$M$ with volumes~$m$ and~$2m$ such that a finite group~$W$ acts on it, and all the conditions of 
Lemma~\ref{peremesh} are held. Then~$M$ is hereditarily normal.
\end{lemma}
\begin{proof}
On the contrary, suppose that~$M$ is not hereditarily normal. Choose a minimal with respect to inclusion ENSS in~$M$. 
Without loss of generality assume that its rank equals the rank of~$M$, and also that the vectors involved in the $\Qg0$-combination constitute the first vectors of the basis of volume~$2m$. Let us show that our NSS consists of one vector of form 
$(\pm 1/2,\ldots,\pm 1/2,0,\ldots,0)$ and several basis vectors. The ENSS obviously contains the vectors which yield~$v_0$ as their semi-sum. Note also that $\ZZ(M)$ has only two cosets modulo the group 
$\ZZ(\bar v_1, \bar v_2,\ldots, \bar v_d)$, since it follows from Lemmas~\ref{proSv} and~\ref{peremesh} that all the vectors, 
having non-integer coordinates, must differ by an integer vector. We need to add at least one representative of the second class, 
and one is enough.

But, if we have a vector $v_{d+1}$ of the form $(\pm 1/2,\ldots,\pm 1/2,0,\ldots,0)$, we can easily obtain 
$v_0=(1/2,\ldots, 1/2,0,\ldots,0)$ by adding several $\bar v_i$s. A contradiction.
\end{proof}


\section{The root system~$B_n$}

The root system 
$B_n$, where $n \geqslant 2$, is formed by vectors $\{ \pm e_i \pm
e_j \, , \pm e_i \,|\, 1 \leqslant i,j \leqslant n,\, i\neq j\}$. 
With respect to the system of simple roots 
$$
\alpha_1=e_1-e_2,\,\alpha_2=e_2-e_3,\ldots,\,\alpha_{n-1}=e_{n-1}-e_n,\, \alpha_n=e_n
$$
the fundamental weights have the form
$$
\pi_1=e_1,\, \pi_2=e_1+e_2,\, \ldots, \,\pi_{n-1}=e_1+\ldots+e_{n-1},\, \pi_n=\frac 12 (e_1+\ldots+e_n).
$$
The root lattice $\Xi=\ZZ^n$. The weight lattice~$\Lambda$ has the form
$$
\Lambda=\{(\ell_1,
\ell_2, \dots, \ell_{n}) \;|\;  2\ell_i \in \ZZ, \ell_i - \ell_j \in \ZZ, i,j = 1,\ldots,n \}.
$$
The Weyl group $W$ acts by permutations on the set of coordinates and by sign change of an arbitrary set of coordinates.
A weight $\lambda=(\ell_1, \ell_2, \dots,
\ell_{n})$ is dominant if and only if ${\ell_1\geqslant  \ldots \geqslant \ell_n\geqslant 0}$.
If all coordinates of $\lambda$ are integers (or all together half-integers but not integers), then the set $M(\lambda)$ consists of all integer (or strictly half-integer, respectively) points in the polytope $P(\lambda)$. 

\subsection{Positive results}\label{sec1}\indent

\slu\label{line1} $\lambda=\pi_1=(1,0,\ldots,0)$.
Then $M(\lambda)=\{\pm e_i \, | \, 1\leqslant i \leqslant n\}$.
Obviously, this subset is unimodular, and by Theorem~$\ref{lemmaunimodsverhnas}$ it is hereditarily normal.

\slu\label{line9} $\lambda=2\pi_2=(1,1)$, $n = 2$. It is easy to check case-by-case that the set $\{ \pm e_1, \pm e_2, \pm e_1\pm e_2 \}$ is 
hereditarily normal.

\slu\label{sluchaj3} $\lambda =\pi_n= \left(\frac{1}{2}, \dots,
\frac{1}{2}\right)$, $2\leqslant n \leqslant4$. After multiplying by~2, we have
$$
M'(\lambda)=\Bigl\{(\underbrace{\pm1,\pm1,\ldots,\pm1}_{n\;\text{coordinates}})\Bigr\}.
$$

For $n=2,3$ the set $M'(\lambda)$ is unimodular, so by Theorem~$\ref{lemmaunimodsverhnas}$ it is hereditarily normal.

Now let $n=4$. The values of all nonzero determinants in~$M'(\lambda)$ equal $\pm 8$ and $\pm 16$. This means that $M'(\lambda)$
is almost unimodular. Let us find all $4$-tuples of vectors of~$M'(\lambda)$ such that their determinant equals~$16$. We may assume that the first vector in this $4$-tuple is~$(1,1,1,1)$. Using case-by-case
consideration, we see that up to multiplying vectors by~$(-1)$, it is the set of rows of the matrix
$$
\begin{pmatrix}w_1\\w_2\\w_3\\w_4\end{pmatrix}=\qmatrix{1&1&1&1\\1&1&-1&-1\\1&-1&1&-1\\1&-1&-1&1}.
$$
Note that the action of~$W$ on $M'(\lambda)$ contains all transpositions of vectors of the form~$\pm w_i$ and~$\pm w_j$, so 
Lemma~\ref{resultm2m} can be applied.

\subsection{Some negative results}

\indent
\prr\label{ex7} $\lambda=2\pi_1=2e_1$, $n = 2$. Consider the following ENSS:
 $v_1 = 2e_1$, $v_2 = e_1+e_2$, $v_3 =e_2$, $v_0=e_1= \frac 12 v_1 = v_2 - v_3$.
Use a discriminating linear function $f=3x_1+4x_2$ (see Section~$\ref{sec:unimod}$), then $f(v_1)=6$, $f(v_2)=7$, $f(v_3)=4$,
$f(v_0)=3$. It is clear that $3$ cannot be represented as a sum of integers $4$, $6$, and $7$.

\prr\label{ex6}  $\lambda=\pi_2=e_1+ e_2$, $n \geqslant 3$. Let $v_1=e_1+e_2$, $v_2=e_1-e_2$, $v_3=e_2-e_3$, $v_4=-e_3$. Then
$v_0=e_1=\frac{1}{2}((e_1+e_2)+(e_1-e_2))=(e_1-e_2)+(e_2-e_3)-(-e_3)$,
but $e_1\not \in \ZZ_{\geqslant 0} (v_1,v_2,v_3,v_4)$. To check that it is an ENSS, one can use the discriminating linear function 
$f=3x_1+x_2-5x_3$. 

\prr\label{ex8} $\lambda=\pi_1+\pi_n=(\frac{3}{2}, \frac{1}{2}, \dots,
\frac{1}{2})$, $n\geqslant 2$. Let
$$
v_1 = \left(\frac{3}{2}, \frac{1}{2}, \dots, \frac{1}{2}\right),\;
v_2 = \left(\frac{3}{2}, -\frac{1}{2}, \dots,-\frac{1}{2}\right),\;
v_3 = \left(\frac{1}{2}, \frac{1}{2}, \dots, \frac{1}{2}\right).
$$
Then $v_0=(1, 0, \dots, 0) = 1/3(v_1 + v_2) = v_1 - v_3$,
and if one considers the first coordinate, it is clear that $v_0 \notin \ZZ_{\geqslant 0}(v_1,
v_2, v_3)$.

\prr\label{ex9} $\lambda = \pi_n = \left(\frac{1}{2}, \dots,
\frac{1}{2}\right)$, $n= 5$. To simplify the notation, multiply all the coordinates by~$2$. Let
\begin{gather*}
\qmatrix{v_1\\ v_2\\ v_3\\ v_4\\ v_5\\ v_6
}=\qmatrix{
1& 1& 1& 1& -1\\
1& 1& 1& -1& 1\\
1& 1& -1& 1& 1\\
1& -1& 1& 1& 1\\
-1& 1& 1& 1& 1\\
1& 1& 1& -1& -1 },\\
v_0 = \frac 13 (v_1+v_2+v_3+v_4+v_5)=(1, 1, 1, 1, 1) = v_1+v_2-v_6.\\
\end{gather*}
Now apply the discriminating function $f = 3x_1+3x_2+3x_3+2x_4+2x_5$.

\subsection{Reduction to the already examined cases}

By a {\it shift for $B_n$} we call the procedure of replacing the vector $\lambda=(\ell_1, \dots, \ell_n)$ with the vector $\lambda'=(\ell_1, \ldots, \ell_i-1, \ldots, \ell_n)$, if 
$\ell_i\geqslant 1$. Notice that $\lambda'$ always belongs to $M(\lambda)$ because 
$\lambda-\lambda'\in\Xi$ and $\lambda'$ is a convex linear combination of vectors $\lambda$ and
$(\ell_1, \ldots, -\ell_i, \ldots, \ell_n)$ with suitable coefficients (these vectors both belong to~$M(\lambda)$).

\begin{lemma}\label{l9}
Let $n\geqslant 3$. If $\lambda \in \Xi\setminus \Phi$, 
then the vector $e_1 + e_2$ belongs to~$M(\lambda)$.
\end{lemma}

\begin{proof}
Let $\lambda = (\ell_1, \dots, \ell_n)$. Since $\lambda$ is dominant, we have $\sum_1^n{\ell_i} \geqslant 2$. 
If $\sum_1^n{\ell_i} > 2$ and $\ell_i> 0$, then the point $(\ell_1, \dots, \ell_{i-1}, \ell_i - 1,
\ell_{i+1}, \dots, \ell_n)$ belongs to $M(\lambda)$ (apply the shift). Repeating this procedure,
we show that there is a point $\lambda'\in M(\lambda)$ with $\sum_1^n{\ell_i'} = 2$. It is either a root
$e_i + e_j$, or $2e_i$, in the second case we can obtain $2e_j$ by acting with~$W$, and the convex hull of $2e_i$ and $2e_j$ contains the point $e_i + e_j$, so $e_i+e_j\in M(\lambda)$, 
hence $e_1+e_2\in M(\lambda)$, as well.
\end{proof}

Now, using Corollary~$\ref{ovlozhenii}$, we show how all cases from~$B_n$, which do not appear in Theorem~$\ref{mth}$, can be reduced to 
Examples $\ref{ex7}$~-- $\ref{ex9}$. If all coordinates of $\lambda$ are integers and $n\geqslant 3$, then every weight $\lambda$ which does not belong to~$\Phi$ 
 can be reduced to $e_1+e_2$ by Lemma~$\ref{l9}$, i.e. Counterexample~$\ref{ex6}$ can be applied. If all coordinates of $\lambda$ are integers and $n=2$, 
then $\lambda = (\ell_1, \ell_2)\ne (2,0)$ but it is not a root, which gives $\ell_1\geqslant 2$, hence $(2,0) \in M(\lambda)$, and we can apply
Corollary~$\ref{ovlozhenii}$ to Counterexample~$\ref{ex7}$.

If all coordinates of $\lambda = (\ell_1, \dots, \ell_n)$ are half-integers, and if in addition there exists an~$i$ such that
$2 \ell_i \geqslant 3$, then~$M(\lambda)$ contains the point $(\frac{3}{2}, \frac{1}{2}, \dots, \frac{1}{2})$ (apply several shifts),
and we can apply Corollary~$\ref{ovlozhenii}$ to Counterexample~$\ref{ex8}$.

Finally, if $\lambda = \left(\frac{1}{2}, \dots, \frac{1}{2}\right)$, then 
after multiplying by~2, 
$M'(\lambda)=\{(\underbrace{\pm1,\pm1,\ldots,\pm1}_{n\;\text{coordinates}})\}$.
For $n=5$ see Counterexample~$\ref{ex9}$, for~$n> 5$ an NSS can be constructed in the following way: take Counterexample~$\ref{ex9}$ for $n=5$
and append to each~$v_i$ $n-5$ coordinates equal to the $5$th coordinate of~$v_i$.

\section{The root system $C_n$} 
The root system $C_n$, $n \geqslant 3$, is formed by vectors $\{ \pm e_i \pm
e_j \, , \pm 2 e_i \,|\, 1 \leqslant i,j \leqslant n,\, i\neq j\}$. 
With respect to the system of simple roots 
$$
\alpha_1=e_1-e_2,\,\alpha_2=e_2-e_3,\ldots,\,\alpha_{n-1}=e_{n-1}-e_n,\, \alpha_n=2 e_n
$$
the fundamental weights have the form
$$
\pi_1=e_1,\, \pi_2=e_1+e_2,\, \ldots, \, \pi_n=e_1+\ldots+e_n.
$$
The root lattice $\Xi=\bigl\{(a_1,\ldots,a_n)\in\ZZ^n \mid \sum_1^n a_i \kratno 2\bigr\}$. The weight 
lattice~$\Lambda=\ZZ^n$.
The Weyl group $W$ acts by permutations on the set of coordinates and by sign changes on an arbitrary subset of coordinates.
A weight $\lambda=(\ell_1, \ell_2, \dots, \ell_{n})$ is dominant if and only if $\ell_1\geqslant  \ldots \geqslant \ell_n\geqslant 0$.
The set $M(\lambda)$ coincides with the set of integer points in $P(\lambda)$ such that the sum of their coordinates has the same parity as~$\lambda$.

\subsection{Positive results} 
\indent
\slu\label{line8} $\lambda=\pi_1=e_1$. Obviously, $M(\lambda)$ is hereditarily normal.

\slu\label{sluchaj5} $\lambda=\pi_2=e_1+e_2$, $n=3,4$.
For $n=3$ $M(\lambda)$ is unimodular, hence it is hereditarily normal.
For $n=4$ it is almost unimodular, since all nonzero determinants are equal to~$\pm 2$ or
$\pm 4$. Without loss of generality, a $4$-tuple of vectors with the determinant~$\pm 4$ coincides with the set of rows of the matrix
$$
\qmatrix{v_1\\v_2\\v_3\\v_4}=
\qmatrix{1&1&0&0\\1&-1&0&0 \\0&0&1&1
\\0&0&1&-1}.
$$
Note that $W$ contains a 4-element subgroup which acts on $\{\langle v_1\rangle ,\langle v_2\rangle ,\langle v_3\rangle ,\langle v_4\rangle \}$ as the Klein four-group , hence, we can apply Lemma~\ref{resultm2m}.

\subsection{Some negative results}
\indent
\prr\label{ex1} $\lambda=\pi_1+\pi_2=(2,1,0)$, $n=3$. Let
\begin{gather*}
\begin{pmatrix}v_1\\v_2\\v_3\\v_4\end{pmatrix}=
\qmatrix{2&1&0\\0&2&1\\1&0&2\\1&2&0},\\
v_0=(1,1,1)=1/3(v_1+v_2+v_3)=v_1+v_2-v_4.
\end{gather*}
We can apply the discriminating function $f=100x_1+10x_2+x_3$.

\prr \label{ex2} Let $\lambda=2\pi_1=2e_1$, $n=3$. Construct an ENSS:
\begin{gather*}
\begin{pmatrix}v_1\\v_2\\v_3\\v_4\end{pmatrix}=
\qmatrix{2&0&0\\0&2&0\\1&0&1\\0&-1&1},\\
v_0 = e_1 + e_2 = 1/2(v_1 + v_2) = v_3 - v_4,\;
f=5x_1+3x_2+9x_3.
\end{gather*}

\prr \label{ex3} Take $\lambda=\pi_3=e_1+e_2+e_3$, $n=3$. Consider the following ENSS:
\begin{gather*}
\begin{pmatrix}v_1\\v_2\\v_3\\v_4\end{pmatrix}=
\qmatrix{1&1&1\\1&-1&-1\\0&1&0\\0&0&-1},\\
v_0=e_1=1/2(v_1+v_2)=v_1-v_3+v_4,\; f=11x_1+6x_2-14x_3.
\end{gather*}

\prr\label{ex4} $\lambda=\pi_4=e_1+e_2+e_3+e_4$, $n=4$. Consider vectors
\begin{gather*}
\begin{pmatrix}v_1\\v_2\\v_3\\v_4\end{pmatrix}=\qmatrix{1&1&1&1\\1&1&-1&-1\\1&0&1&0\\0&-1&1&0}.
\end{gather*}
Consider $v_0=(1,1,0,0)=\frac12(v_1+v_2)=v_3-v_4$ and the discriminating function
$f=5x_1+5x_2+8x_3-x_4$. Then $f(v_1)=17$, $f(v_2)=f(v_4)=3$,
$f(v_3)=13$, $f(v_0)=10$. Since $f(v_1)$ and $f(v_3)$ are too big, $v_1$ and $v_3$ cannot be used in a $\ZZ_{\geqslant 0}$-combination.
But~$10$ is not divisible by~$3$, and we cannot obtain~$v_0$, using only $v_2$ and $v_4$.

\prr\label{ex5} Let $\lambda=\pi_2=e_1+e_2$, $n=5$. Consider an ENSS:
\begin{gather*}
\begin{pmatrix}v_1\\v_2\\v_3\\v_4\\v_5\\v_6\end{pmatrix}=
\qmatrix{1&0&1&0&0\\1&0&-1&0&0\\0&1&0&1&0\\0&1&0&-1&0\\0&0&1&0&1\\0&0&0&1&1},\\
v_0=e_1+e_2=1/2(v_1+v_2+v_3+v_4)=v_2+v_3+v_5-v_6,\\
f=5x_1+6x_2+x_3+2x_4+20x_5.
\end{gather*}

\begin{remark}\label{udlinim}
Counterexamples $\ref{ex1}$~--- $\ref{ex3}$ work for all $n\geqslant 3$, Counterexample~$\ref{ex4}$ works for all $n\geqslant 4$, and 
Counterexample~$\ref{ex5}$ works
for all $n\geqslant 5$. Indeed, we can append $n-3$ ($n-4$ and $n-5$, respectively) zero coordinates to each vector.
\end{remark}

\subsection{Reduction to the already examined cases}

Consider two cases: a) all $ \ell_i\in\{0, 1\}$; b) there is at least one $\ell_i$ with $|\ell_i |\geqslant 2$.

First consider case~a): all $\ell_i\in\{0,1\}$, which means that $\lambda=\pi_k=e_1+e_2+\ldots+e_k$,
$k\leqslant n$.

\begin{lemma}\label{+n}
An NSS for the pair $(k,n_0)$ serves as an NSS for all the pairs $(k,n)$, where $n\geqslant n_0$.
\end{lemma}
\begin{proof}
Append $n-n_0$ zero coordinates to each vector.
\end{proof}

\begin{lemma}\label{k+2}
An NSS for the pair $(k,n)$, where $k+2\leqslant n$, is also an NSS for the pair $(k+2, n)$.
\end{lemma}
\begin{proof}
If $\lambda=e_1+\ldots+e_{k+2}$, then
$$
e_1+e_2+\ldots+e_k=\lambda-(e_{k+1}+e_{k+2})=\frac12 (\lambda+(e_1+\ldots+e_k-e_{k+1}-e_{k+2})),
$$ 
hence it belongs to~$M(\lambda)$. 
Applying Corollary~$\ref{ovlozhenii}$, we obtain that an NSS for $(k,n)$ is also an NSS for $(k+2,n)$.
\end{proof}

Now take a pair $(k,n)$, not equal to $(1,n)$, $(2,2)$, $(2,3)$, and $(2,4)$, where $k\leqslant n$.

If $k$ is even and $n\leqslant 4$, then it is the pair $(4,4)$, i.e. it is Counterexample~$\ref{ex4}$.
If $k$ is even and $n\geqslant 5$, then we can modify Counterexample~$\ref {ex5}$ for the pair $(2,5)$ to get the required NSS: 
firstly apply Lemma~$\ref{+n}$, and then apply Lemma~$\ref{k+2}$.
If $k$ is odd and $k\geqslant 3$, then we can modify Counterexample~$\ref{ex3}$ for the pair $(3,3)$ to get the required NSS in the same way.

\medskip
Now consider Case~b).

\begin{definition} By a {\it shift for $C_n$} we denote the procedure of replacing the point 
$\lambda=(\dots,l,\dots,l',\dots)$ 
with the point $\lambda'=(\dots,l-1,\dots,l'+1,\dots)$ (at the same places) when $l-l' \geqslant 2$.
\end{definition}

It is easy to see that the point $\lambda'$ belongs to $M(\lambda)$. 
Also, if we consequently apply steps, then their number is finite.

\begin{lemma}\label{l3.5}
Let $\lambda=(\ell_1,\ldots, \ell_n)$, and $\exists\, i$ such that $\ell_i\geqslant 2$. Then $M(\lambda)$ contains either $(2,0,\ldots,0)$ or $(2,1,0,\ldots,0)$.
\end{lemma}

\begin{proof}
Since $\lambda$ is dominant, we have $\ell_1\geqslant 2$. Now change $\lambda$, during this process we let it be non-dominant. 
Change sign at~$\ell_n$ in such a way that $\ell_n\leqslant 0$, and shift it with $\ell_1$ several times till the moment when $\ell_1$ attains the 
value~$2$. If meanwhile $\ell_n$ becomes positive, then change its sign to make it negative, and so on. Then fix $\ell_1=2$ and shift other coordinates 
in any possible way, changing signs at some coordinates, if needed. This process is finite, and if no shifts are possible, then it is either the point~$(2,0,\ldots,0)$, or the point~$(2,1,0,\ldots,0)$.
\end{proof}

Now we can apply Corollary~$\ref{ovlozhenii}$ to Counterexamples~$\ref{ex1}$ and~$\ref{ex2}$. We obtain that in Case~b) 
there exists an NSS for every~$\lambda$.

\section{The root system $D_n$}
The root system 
$D_n$, $n \geqslant 4$, consists of vectors $\{ \pm e_i \pm e_j  \,|\, 1 \leqslant i,j \leqslant n,\, i\neq j\}$. 
With respect to the system of simple roots 
$$
\alpha_1=e_1-e_2,\,\alpha_2=e_2-e_3,\ldots,\,\alpha_{n-1}=e_{n-1}-e_n,\, \alpha_n=e_{n-1}+e_n
$$
the fundamental weights have the form
\begin{gather*}
\pi_1=e_1,\, \pi_2=e_1+e_2,\, \ldots, \,\pi_{n-2}=e_1+\ldots+e_{n-2}, \\
 \pi_{n-1}=\frac 12 (e_1+\ldots+e_{n-1}-e_n),\, \pi_n=\frac 12 (e_1+\ldots+e_{n-1}+e_n).
\end{gather*}
The root lattice $\Xi=\bigl\{(a_1,\ldots,a_n)\in\ZZ^n \mid \sum_1^n a_i \kratno 2\bigr\}$. The weight lattice~$\Lambda$ has the form
$$
\Lambda=\{(\ell_1,
\ell_2, \dots, \ell_{n}) \;|\;  2\ell_i \in \ZZ, \ell_i - \ell_j \in \ZZ, i,j = 1,\ldots,n \}.
$$
The Weyl group $W$ acts by permutations on the set of coordinates and by sign changes on any set of coordinates of even cardinality.
A weight $\lambda=(\ell_1, \ell_2, \dots, \ell_{n})$ is dominant if and only if $\ell_1\geqslant  \ldots \geqslant \ell_n$, $\ell_{n-1}+\ell_n\geqslant 0$.
If all the coordinates of $\lambda$ are integers (strictly half-integers), then the set $M(\lambda)$ consists of all integer 
(strictly half-integer)
points in the polytope $P(\lambda)$, such that their sum of coordinates differs with the sum of coordinates of~$\lambda$ by an even number.

The reasoning for~$D_n$ has another structure than for~$B_n$ and~$C_n$. The cases of integer and half-integer coordinates of the highest weight 
are considered separately. Many NSSs are taken from Case~$C_n$. The {\it Shift for $D_n$} is the same as the Shift for~$C_n$.

\subsection{Coordinates of all weights are integers}

\indent
\slu\label{line2} $\lambda=\pi_1=e_1$. Obviously, $M(\lambda)$ is hereditarily normal.

\slu\label{line12} $\lambda=\pi_2=e_1+e_2$,
$n=4$. The set $M(\lambda)$ coincides with the analogous set from Case~$\ref{sluchaj5}$, hence it is hereditarily normal.

Let us construct NSSs in all the other cases. To use NSSs constructed for~$C_n$, it is only necessary to check that 
the weights under consideration belong to~$M(\lambda)$ for~$D_n$. If a point~$v$ has a zero coordinate, then its orbits under the Weyl groups in cases $C_n$ and $D_n$ 
coincide, because the zero coordinate can be, if needed, multiplied by $-1$.

\prr\label{example10} $\lambda=\pi_1+\pi_2=(2,1,0,0)$. We can use Counterexample~$\ref{ex1}$.

\prr\label{example11} $\lambda=2\pi_1=(2,0,0,0)$, $n=4$. Counterexample~$\ref{ex2}$ with the appended column of zeroes works.

\prr\label{ex10} $\lambda=\pi_3+\pi_4=e_1+e_2+e_3$, $n=4$. 
Counterexample~$\ref{ex3}$ with the appended column of zeroes works.

\prr\label{ex12} $\lambda=\pi_2=e_1+e_2$, $n=5$. Counterexample~$\ref{ex5}$ can be applied.

\prr\label{ex11} $\lambda=2\pi_4=e_1+e_2+e_3+e_4$, $n=4$. Counterexample~$\ref{ex4}$ can be applied.

\medskip
Now, using Counterexamples above, let us show that NSSs exist if a) all nonzero coordinates
of the highest weight equal $\pm 1$ and their sum is odd; b) all nonzero coordinates of the highest weight equal $\pm 1$ and their sum is even; 
c) $\lambda$ has a coordinate such that its absolute value is not less than $2$.

In this subsection all coordinates are integers, consequently, every set of weights for every~$n$ can be considered as a set of weights for a 
greater $n$, if we fill new coordinates with zeroes. 
Hence, Counterexamples~\ref{example10}, \ref{example11}, \ref{ex10}, and~\ref{ex11} provide us with NSSs for highest weights of the same form for all $n\geqslant 4$, and Counterexample~\ref{ex12}~--- for all $n\geqslant 5$.

In case~a), there are either~$3$ nonzero coordinates, and it is Counterexample~\ref{ex10}, or at least~$5$ nonzero coordinates. 
We can make two last of them zero: take $\lambda'$ which differs from $\lambda$ by the signs of two last coordinates, and replace~$\lambda$ 
with the midpoint of the interval $\lambda\lambda'$. Then treat two more coordinates, etc., and finally we reduce this case to 
Counterexample~$\ref{ex10}$.

In case b), if we have only two nonzero coordinates, we can obtain an NSS from Counterexample~$\ref{ex12}$:
just append the required number of zeroes. If there are $4$ nonzero coordinates, then an NSS can be obtained from Counterexample~$\ref{ex11}$
by appending the required number of zeroes. If there are more than $4$ nonzero coordinates (recall that each equals~$\pm 1$), 
then make two last of them zero, then two more, and repeat this procedure up to the moment when their number equals $4$, and then use Corollary~$\ref{ovlozhenii}$.

In case c), depending on the parity of $\sum_1^{n}\ell_i$,
one has to show that~$M(\lambda)$ contains either the point $(2,0\ldots,0)$ or the point $(2,1,0,\ldots,0)$. 
Firstly change $\lambda$ to a point having at least one zero coordinate: let $\lambda'$ be the vector obtained from~$\lambda$ by 
the simultaneous sign change of the two last coordinates,
then apply the Shift to the second and the last but one coordinates of~$\lambda'$ up to the moment when the last but one coordinate attains zero. 
Permute $n-1$ last coordinates to make $\ell_n=0$ and apply the algorithm from the proof of Lemma~\ref{l3.5} of Case~$C_n$ to $n-1$ first 
coordinates of~$\lambda'$. If during this process we acted by sign changes of an odd number of indices for $C_n$, the same is possible for $D_n$: 
due to the form of~$\lambda'$ we can also change the sign of the zero coordinate.

\subsection{Coordinates greater than~1}
In this section we suppose that coordinates of all weights are nonintegers and that the highest weight has a coordinate whose absolute value is not less than $\frac 32$.

\begin{lemma}
Under the conditions formulated above, $M(\lambda)$ contains a point of the form $\left(\frac
32,\frac 12, \frac 12, \ell'_4,\ell'_5,\ldots,\ell'_n\right)$, where $\ell'_i$ are half-integers, $i=4,\ldots,n$.
\end{lemma}

\begin{proof}
Since $\lambda=(\ell_1,\ldots,\ell_n)$ is dominant, $\ell_1$ is one of the coordinates with the maximal absolute value.
Now change $\lambda$, letting it be nondominant.
If $\ell_1> 3/2$, 
we need to replace~$\lambda$ with a point having negative coordinates (for this it suffices to change the signs of two last coordinates) and then shift 
$\ell_1$ with any negative coordinates till the moment when~$\ell_1$ attains~$3/2$. 
Now fix $\ell_1$ and perform the same procedure with~$\ell_2$ till the moment when $\ell_2=1/2$. If now~$\ell_3$ and~$\ell_4$ have the same sign, change signs at $\ell_2$ and $\ell_4$ and then shift~$\ell_3$ and~$\ell_4$ till
the moment when one of them becomes~$\pm 1/2$. 
Permuting the 
coordinates, if needed, we may suppose that we obtained the point~$(3/2, \pm 1/2, \pm 1/2,\ldots)$. Now, if necessary, change the signs of the pairs of 
coordinates~$2,4$ and $3,4$ and obtain the point of the required form.
\end{proof}

\prr For $\lambda=\left(\frac 32,\frac 12, \frac 12, \ell'_4,\ell'_5,\ldots,\ell'_n\right)$ consider the following ENSS:
\begin{gather*}
v_1=\left(\frac 32,\frac 12, \frac 12, \ell'_4,\ell'_5,\ldots,\ell'_n\right),\;
v_2=\left(-\frac 12,-\frac 32, \frac 12, \ell'_4,\ell'_5,\ldots,\ell'_n\right),\;
v_3=\left(\frac 12,\frac 32, \frac 12, \ell'_4,\ell'_5,\ldots,\ell'_n\right),\\
v_4=\left(-\frac 12,\frac 12, \frac 12,\ell'_4,\ell'_5,\ldots,\ell'_n \right)=\frac 12 \left( \left(-\frac 32,-\frac 12,
\frac 12, \ell'_4,\ell'_5,\ldots,\ell'_n\right)+\left(\frac 12,\frac 32, \frac
12, \ell'_4,\ell'_5,\ldots,\ell'_n\right) \right).
\end{gather*}
Then $v_0=\left(\frac 12,-\frac 12, \frac 12,
\ell'_4,\ell'_5,\ldots,\ell'_n\right)=\frac 12 (v_1+v_2)=v_1+v_4-v_3$. In view of the third coordinate, it is clear that it is indeed an~NSS.

\subsection{Coordinates smaller than~1} Now we suppose that all the coordinates of the highest weight are nonintegers and that their absolute value is less than~$1$, i.e. $\lambda=(1/2,\ldots,1/2,\pm 1/2)$, $\lambda \in \{\pi_{n-1},\pi_n\}$. We assume that $\lambda=\pi_n$.

\slu\label{line4} For $n=4$ the set $M(\lambda)$ is a subset of~$M(\pi_4)$ for $B_4$ (see Case~$\ref{sluchaj3}$). Since in the case of~$B_4$ all the subsets are saturated, here it is also true.

Now the aim is to show that for $n=5,\, 6$ the answer is positive, and for $n\geqslant 7$ it is negative. 

\slu\label{line6}  $\lambda=\pi_5=\left(\frac 12,\frac 12,\frac
12,\frac 12,\frac 12\right)$, $n=5$. 
After multiplying by~2
$$
M'(\lambda)=\{(\pm 1,\pm
1,\pm 1,\pm 1,\pm 1)\,|\,\text{even number of minuses}\}.
$$
Let us show that $M'(\lambda)$ is almost unimodular of volume~$16$. To compute the determinant of five arbitrary vectors, write them 
as a matrix and add the first row of this
matrix to all the other rows. Now rows $2$--$5$ are even, hence the volume of the determinant is divisible by~$16$. For the following vectors 
$$
\qmatrix{1&1&1&1&1\\1&-1&-1&-1&-1\\-1&1&-1&-1&-1\\-1&-1&1&-1&-1\\-1&-1&-1&1&-1}
$$
the determinant equals $16$, hence $M'(\lambda)$ is almost unimodular. Notice that every vector has length~$\sqrt 5$. The value of the determinant is at the
same time the volume of the parallelepiped generated by these vectors, and the absolute value of the last number does not exceed ${(\sqrt 5)^5<64}$, 
hence equals $16$, $32$, or $48$.

Letting $m=16$, we obtain that all possible nonzero values of determinants are $\pm m$, $\pm 2m$, or $\pm 3m$.

\begin{lemma}\label{new1}
If for some vectors $v_1,\ldots,v_5\in M'(\lambda)$ the scalar product $(v_1,v_2)=-3$, then 
$$
|\det(v_1,\ldots,v_5)|<3m.
$$
\end{lemma}
\begin{proof}
Each vector from $M'(\lambda)$ has length $\sqrt 5$. Let $S_{12}$ be the area of the parallelogram generated by vectors $v_1$ and $v_2$. 
Since $(v_1,v_2)=-3$, we have $S_{12}=4$. From geometrical reasons $|\det(v_1,\ldots,v_5)|\leqslant S_{12}\cdot (\sqrt 5)^3<48=3m$.
\end{proof}

\begin{lemma}\label{new2}
Let distinct vectors $v_1,\ldots,v_6\in M'(\lambda)$ be such that $|\det(v_1,\ldots,v_5)|=3m$.
\begin{enumerate}
\item 
Up to the permutation of lines and up to the simultaneous sign change in pairs of~columns,
$$
\qmatrix{
v_1 \\
v_2 \\
v_3 \\
v_4 \\
v_5
}
=
\qmatrix{
1&1&1&1&1 \\
-1&-1&1&1&1 \\
-1&1&-1&1&1 \\
-1&1&1&-1&1 \\
-1&1&1&1&-1}.
$$
\item 
$|\det(v_1,v_2,v_3,v_4,v_6)| \neq 3m$.
\end{enumerate}
\end{lemma}
\begin{proof}
(i) It follows from Lemma~$\ref{new1}$ that no two of these vectors differ in four coordinates. Hence, any two of these vectors differ exactly in~$2$ coordinates.
Without loss of generality $v_1=(1,1,1,1,1)$ and $v_2=(-1,-1,1,1,1)$. Then each of the three other vectors has exactly two $(-1)$s. Say that two first coordinates 
are {\it prefix}. To differ from $v_2$ in exactly two coordinates, each of the remaining vectors must have exactly one prefix coordinate equal to~$(-1)$. 
It follows from the pigeonhole principle that two of them (say, $v_3$ and $v_4$) have the same prefix coordinate equal to~$(-1)$, without loss of 
generality this is the first coordinate. Then the first coordinate of $v_5$ also equals $(-1)$, otherwise $v_5$ cannot differ simultaneously with $v_2$, $v_3$, and~$v_4$ in 
two coordinates. Since all the vectors are pairwise distinct, we obtain the same set as in the formulation of the Lemma.

(ii) Suppose the contrary. Then it follows from Lemma~$\ref{new1}$ that $v_6$ has to have two $(-1)$s. Now in the 5-tuple $(v_1,v_2,v_3,v_4,v_6)$ the vector $v_6$ cannot simultaneously differ with $v_2$, $v_3$, and~$v_4$ in two coordinates: 
up to symmetry, it is either $(1,-1,-1,1,1)$, or $(1,-1,1,1,-1)$.
\end{proof}

\begin{lemma}\label{lemma2m}
Let $v_1,\ldots,v_6\in M'(\lambda)$ be such that all the absolute values of their nonzero determinants are greater than $m$. 
Then all these determinants equal~$\pm 2m$.
\end{lemma}
\begin{proof}
On the contrary, suppose that there is a determinant equal to~$\pm 3m$. It follows from Lemma~$\ref{new2}$ that all the other nonzero determinants equal 
$\pm 2m$. But the alternating sum of six determinants of $5$-tuples of our vectors equals $\det(v_1-v_2,v_1-v_3,\dots,v_1-v_6)$. In the corresponding matrix
all the entries are even, hence the determinant is divisible by $32=2m$. Contradiction with the fact that $3m\pm 2m\pm \ldots\pm 2m$ is not divisible by~$2m$.
\end{proof}

Consider an ENSS $\{v_0; v_1,v_2,\ldots,v_s\}$. If the rank~$d$ of this set is less than~$5$, then add ${5-d}$ vectors from $M'(\lambda)$ to make the rank equal to~$5$. 
Now suppose that this ENSS is
$\{v_0; v_1,v_2,v_3,v_4, v_5,\ldots,v_s\}$, and only
$v_1,v_2,v_3,v_4,v_5$ appear in the $\Qg0$-combination
(maybe with zero coefficients). We may compute all the determinants of the form
$\det (v_1,\ldots,\widehat v_i,\ldots, \\ v_5,v_j)$, where one of the first $5$ vectors is thrown out and one new vector is added instead of it. 
Case~a): one of these determinants equals~$\pm m$, case b): for every nonzero determinant its absolute value is greater than~$m$.

In case b) it follows from Lemma~$\ref{lemma2m}$ that we have $s-5$ unimodular six-element subsets 
$\{v_1,\ldots,v_5,v_j\}$, $6\leqslant j\leqslant s$,  
of volume $m'=2m$. In each of them $v_j$ can be expressed in $v_1,\ldots,v_5$ with integer coefficients, 
hence the determinant of each $5$-tuple in the set $\{v_1,\ldots,v_s\}$ is divisible by $2m$, hence equals~$\pm 2m$. This ENSS
is hereditarily normal by Theorem~$\ref{lemmaunimodsverhnas}$, a contradiction.

Case a) needs more punctuality. It follows from Lemma~$\ref{nemin}$ that the determinant~$\pm m$ does not coincide with
$\det(v_1,v_2,\ldots,v_5)$. Without loss of generality $\det(v_1,\ldots,v_4,v_6)=16$ (if it equals $-16$, then transpose two first vectors, and 
the determinant will change sign). By our assumption $\det(v_1,\ldots,v_5)=\pm 2m \text{ or }\pm 3m$.

\begin{lemma}\label{l17}
There are no vectors $w_1,\ldots,w_6$ in $M'(\lambda)$ such that the following is true (simultaneously): $\det (w_1,\ldots,w_5)=\pm 2m$, 
$\det (w_5,w_2,w_3,w_4,w_6)=\pm 2m$, these determinants have different signs, and $\det(w_1,\ldots,w_4,w_6)=\pm m$.
\end{lemma}
\begin{proof}
Straightforward check using software Maple~7, \cite{Maple}. 
\end{proof}

\begin{lemma}\label{l18}
There are no vectors $w_1,\ldots,w_6$ in $M'(\lambda)$ such that 
$$
\det (w_1,\ldots,w_5)=-2m \text{ and } \det(w_1,\ldots,w_4,w_6)=-3m.
$$
\end{lemma}
\begin{proof}
Using Lemma~$\ref{new2}$, we may assume that $w_6=(1,1,1,1,1)$ and
$$
\qmatrix{
w_1 \\
w_2 \\
w_3 \\
w_4 
}
=
\qmatrix{
-1&-1&1&1&1 \\
-1&1&-1&1&1 \\
-1&1&1&-1&1 \\
-1&1&1&1&-1}. 
$$
The hyperplane $\langle w_1,w_2,w_3,w_4 \rangle$ is defined by the equation $2x_1+x_2+x_3+x_4+x_5=0$. Since $\det (w_1,\ldots,w_5)<0$ 
and $\det (w_1,\ldots,w_6)<0$, the vectors $w_5$ and $w_6$ belong to the same half-space with respect to this hyperplane. Hence,
exactly two coordinates of $w_5$ equal~$-1$. Without loss of generality $w_5=(1,-1,-1,1,1)$, but the corresponding determinant
equals~$-16=-m$, a contradiction.
\end{proof}

Now let us re-consider the ENSS. Analyze the following decompositions in basis: $v_0$ in the basis $v_1,\ldots,v_5$ 
and in the basis $v_1,\ldots,v_4,v_6$, and $v_5$ in the basis $v_1,\ldots,v_4,v_6$. Let
\begin{gather*}
v_0=q_1v_1+\ldots+q_5v_5 \mbox{ be the initial } \QQ_{\geqslant
0}\mbox{-combination, and }\\
v_0=z_1v_1+\ldots+z_4v_4+z_6v_6 \mbox{ be the initial }
\ZZ\mbox{-combination,}\\
v_5=y_1v_1+\ldots+y_4v_4+y_6v_6, \quad y_i\in\{0,\pm 1, \pm 2, \pm
3\}\; -
\end{gather*} 
just the decomposition in the basis. Then 
\begin{gather*}
v_6=-\frac{y_1}{y_6}v_1-\ldots-\frac{y_4}{y_6}v_4+\frac1{y_6}v_5\quad
\Rightarrow \\
v_0=z_1v_1+z_2v_2+z_3v_3+z_4v_4+z_6\Bigl(-\frac{y_1}{y_6}v_1-\ldots-\frac{y_4}{y_6}v_4+\frac1{y_6}v_5\Bigr).
\end{gather*}
From the uniqueness of the decomposition in the basis it follows that
\begin{gather*}
q_1=z_1-z_6\frac{y_1}{y_6},\ldots,q_4=z_4-z_6\frac{y_4}{y_6}, 
q_5=z_6\frac1{y_6}, \mbox{ all }q_i\in[0,1).
\end{gather*}

If $|y_6|=3$, i.e. $|\det(v_1,v_2,\ldots,v_5)|=3m$, then by Lemma~$\ref{new2}$
$$
\qmatrix{
v_1 \\
v_2 \\
v_3 \\
v_4 \\
v_5
}
=
\qmatrix{
1&1&1&1&1 \\
-1&-1&1&1&1 \\
-1&1&-1&1&1 \\
-1&1&1&-1&1 \\
-1&1&1&1&-1}. 
$$
The linear combination of these vectors with $\Qg0$-coefficients 
$q_1,\ldots$, $q_5$ belongs to the weight lattice multiplied by two, hence, all the coordinates of the resulting vector have the same parity. Subtracting
the third coordinate from the second one, we obtain that $2(q_2-q_3)$ is even, which implies
$q_2=q_3$, and analogously $q_2=q_3=q_4=q_5$. The first coordinate of~$v_0$ equals $q_1-4q_2$, while all the others 
equal $q_1+2q_2$. These numbers also have the same parity, consequently, $q_2\in\{0,\frac 13, \frac 23\}$. Since $q_1-4q_2$ and $q_1+2q_2$ are both integers
and cannot simultaneously equal $0$, we obtain that $q_1=q_2\in\{\frac 13, \frac
23\}$. Hence, $v_0$ equals either $(-1,1,1,1,1)$ or $(-2,2,2,2,2)$.

\begin{lemma}
Let $(v_1,v_2,v_3,v_4,v_5)$ be from Lemma~$\ref{new2}$, and let $v_6\in M'(\lambda)\setminus \{v_1,\ldots,v_5\}$. 
Then the vector $(-1,1,1,1,1)$ can be represented as a $\ZZ_{\geqslant
0}$-combination of vectors $v_1$, $v_2$, $v_3$, $v_4$, $v_5$, $v_6$.
\end{lemma}

\begin{proof}
Up to a permutation of indices, $v_6$ is either $(1,1,1,-1,-1)$, or $(1,-1,-1,-1,-1)$, or $(-1,-1,-1,-1,1)$.
Consider these cases separately.

\noindent a)
$v_6=(1,1,1,-1,-1)$. Then
$$
(-1,1,1,1,1)=(1,1,1,-1,-1)+(-1,-1,1,1,1)+(-1,1,-1,1,1).
$$

\noindent b)
$v_6=(1,-1,-1,-1,-1)$. Then
\begin{multline*}
(-1,1,1,1,1)=2(1,-1,-1,-1,-1)+(1,1,1,1,1)+ (-1,-1,1,1,1)+\\
+(-1,1,-1,1,1)+(-1,1,1,-1,1)+(-1,1,1,1,-1).
\end{multline*}

\noindent c)
$v_6=(-1,-1,-1,-1,1)$. Then $(-1,1,1,1,1)=(-1,-1,-1,-1,1)+(1,1,1,1,1)+(-1,1,1,1,-1)$. 
\end{proof}

It remains to consider the case $|y_6|=2$. Here all $q_i\in\{0,\frac12\}$.

If $z_6$ is even, then $q_1=z_1-z_6\frac{y_1}{y_6}=z_1-y_1\frac{z_6}{y_6}$ is an integer from the interval $[0,1)$, hence it equals $0$. 
Analogously all the other $q_i$s, $i=2,\ldots,5$, equal~$0$, consequently, $v_0=0$. A contradiction.

If $z_6$ is odd, then to check the saturation property we will seek for a $\ZZ_{\geqslant 0}$-combination of the following form: $v_0=v_6+n_1v_1+n_2v_2+\ldots+n_5v_5$, $n_i\in\ZZ_{\geqslant 0}$. To prove its existence, let us show that if we decompose $v_0$ 
and $v_6$ in the basis $v_1,\ldots,v_5$, then the corresponding coordinates differ by integer values and that the 
coordinates of~$v_6$ are strictly less than~$1$. Consequently, they will not exceed the corresponding 
coordinates of~$v_0$, since we know that the coordinates of~$v_0$ equal $q_i$ and belong to the interval $[0,1)$.

It is clear from the formulae that the cases $i=1,2,3,4$ and $i=5$ should be considered differently. Since cases $i=1,2,3,4$ are symmetrical,
consider only the cases $i=1$ and $i=5$. Since $\frac{z_6-1}{y_6}$ is integer, we have 
$$
q_1-\left(-\frac{y_1}{y_6}\right)=z_1-z_6\frac{y_1}{y_6}+\frac{y_1}{y_6}=z_1+\frac{(1-z_6)y_1}{y_6}
$$
is integer, analogously $q_5-\frac1{y_6}=\frac{z_6-1}{y_6}$ is integer, i.e., all the differences of the corresponding coordinates are integer.
We also know that
$y_1=\det(v_5,v_2,v_3,v_4,v_6)/m$ and $y_6=\det(v_1,v_2,v_3,v_4,v_5)/m$,  
which means that $|y_1|\in\{0,1,2,3\}$ and
$|y_6|=2$.
It follows from Lemmas~$\ref{new2}$, $\ref{l17}$, and~$\ref{l18}$ that the number~$-\frac{y_1}{y_6}$ is neither~$1$ 
nor~$\frac32$. In all the other cases the inequality $-\frac{y_i}{y_6}<1$ is held for all~$i$, $1\leqslant i\leqslant 4$. It is also clear that $\frac1{y_6}<1$.
Hence, after adding some $v_i$s, $1\leqslant i\leqslant 5$, we can obtain $v_0$ from $v_6$, and the ENSS under consideration is not an ENSS. 
Therefore~$M'(\lambda)$ is hereditarily normal.

\smallskip

\slu\label{line7} $\lambda=\pi_6=\left(\frac 12,\frac 12,\frac
12,\frac 12,\frac 12,\frac 12\right)$, $n=6$. After multiplying by~2, 
$$
M'(\lambda) =\{(\pm 1,\pm 1,\pm 1,\pm 1,\pm 1,\pm 1)\,|\,\text{even number of minuses}\}. 
$$

\begin{lemma}
The set $M'(\lambda)$ is almost unimodular of volume~$64$. The values of determinants equal $\pm 64$ and $\pm 128$, or, equivalently, $\pm m$ 
and $\pm 2m$.
\end{lemma}

\begin{proof}
Consider a subset $\{v_1,v_2,\ldots,v_6\}\subseteq M'(\lambda)$. Without loss of generality we have $v_1=(1,1,1,1,1,1)$. Add $v_1$ to each of the
other vectors and write down the obtained $6$ vectors as the rows of a matrix. The rows from the second till the sixth are even, hence the determinant
is divisible by $32$, and if we divide the rows from the second till the sixth by~$2$, the number of $1$s in each of the rows of the remaining matrix 
will be even. Now add to the first column of the new matrix the sum of all other columns. The new first column is even, hence the determinant of the 
original matrix is divisible by~$64$. 

Now bound it from above. Split the vectors in three pairs and generate a parallelogram with each pair,
then the volume of the parallelepiped does not exceed the product of areas of 
these three parallelograms. Each vector in $M'(\lambda)$ has length~$\sqrt 6$, the absolute value of the scalar product of two arbitrary vectors equals
~$2$, hence the area of each parallelogram equals $6^{2/2}\sqrt{1-(1/3)^2}=2^{5/2}$. Finally, the volume does not exceed
 $2^{15/2}<192$, consequently, its absolute value equals~$64$ or~$128$.
\end{proof}

Suppose that we have an~ENSS $\{v_0; v_1,v_2,v_3,v_4,v_5,v_6\}$ in~$M'(\lambda)$. 
Consider a $\Qg0$-combina\-tion corresponding to the vector $v_0$.
By Lemma~$\ref{nemin}$ we know that $|\det(v_1,v_2,v_3,v_4,v_5,v_6)|$ equals~$128$, consequently, by Lemma~$\ref{kramer}$ 
the coefficients of the $\Qg0$-combination equal~$0$ or~$\frac 12$.

For conveniency, we suppose till the end of this proof that $M'(\lambda)$ consists of points of the form $(\pm 1,\pm 1,\pm 1,\pm
1,\pm 1,\pm 1)$ having odd number of~$(-1)$s. 

\begin{lemma}
If $\det(v_1,v_2,\ldots,v_6)=128$,
then one of the sets $\{\pm v_1, \pm v_2, \pm v_3, \pm v_4, \pm v_5, \pm v_6\}$ (unordered) can be mapped via~$W$ to the set 
$$
\qmatrix{
w_1 \\
w_2 \\
w_3 \\
w_4 \\
w_5 \\
w_6
}
=
\qmatrix{
-1 & 1  & 1  & 1  & 1  &  1 \\
1  & -1 & 1  & 1  & 1  &  1 \\
1  & 1  & -1 & 1  & 1  &  1 \\
1  & 1  & 1  & -1 & 1  &  1 \\
1  & 1  & 1  & 1  & -1 & 1 \\
1  & 1  & 1  & 1  & 1  & -1 }.
$$
\end{lemma}

\begin{proof}
The set $M'(\lambda)$ contains a vector $-v$ for each vector $v$. Hence, to compute the determinants, we may consider only $16$ vectors
instead of $32$, namely those which have one~$(-1)$ and those which have three~$(-1)$s but not on the first position. By the pigeonhole 
principle, there exists a sign change $\sigma \in W$ such that the 6-tuple $\{\sigma v_1,\sigma v_2,\ldots,\sigma v_6\}$ contains at least 3 vectors 
with one or five~$(-1)$s (otherwise $6\cdot 12 < 2\cdot 32$). After this sign change, we may assume that the given $6\times 6$ minor contains vectors~$w_1$, $w_2$, and~$w_3$. The direct check in Maple~7, \cite{Maple}, shows that only the following $6\times 6$ minor satisfies the condition: $(w_1,w_2,w_3,w_4,w_5,w_6)$. 
\end{proof}

For every pair of the form $(\pm w_i, \pm w_j)$ the group~$W$ contains an element interchanging these two vectors, hence Lemma~\ref{resultm2m} 
can be applied.

\medskip
For $n\geqslant 7$ construct an NSS. Multiply all the coordinates by~$2$. After this all the coordinates of the initial vectors become~$\pm 1$.

\prr\label{ex20} Consider vectors
$$
\begin{pmatrix}v_1\\v_2\\v_3\\v_4\\$\,$\\v_5\\v_6\\v_7\end{pmatrix}=
\qmatrix{
1&1&1&1&1&1&1\\1&1&1&-1&-1&-1&-1\\
1&-1&-1&1&1&-1&-1\\1&-1&-1&-1&-1&1&1\\\\1&1&-1&-1&1&1&1\\
1&1&1&1&-1&-1&1\\1&-1&1&1&1&1&-1}.
$$
Then $v_0=(2,0,0,0,0,0,0)=\frac12(v_1+v_2+v_3+v_4)=v_5+v_6+v_7-v_1$.
Let us consider the first coordinate. If $v_0$ is a
$\ZZ_{\geqslant 0}$-combination of some $v_i$s, then it is the sum of exactly two $v_i$s. But no pairwise sum equals $v_0$, hence, it is indeed an NSS.

Counterexample~\ref{ex20} can be easily modified for the greater values of~$n$. Indeed, append $n-7$ coordinates equalling~$1$ to each vector.
It is easy to see that these vectors also belong to $M(\lambda)$ for $\lambda=\pi_{n}$ for all $n\geqslant 7$. Since~$1$ 
is at the same time the first coordinate of all~$v_i$s, every linear combination of $v_i$s will have the same value on each appended coordinate as 
on the first coordinate.

\medskip

Theorem~$\ref{mth}$ is proved.

\end{document}